\documentclass[a4paper, 11pt]{amsart}

\usepackage{amsmath, amsthm, amssymb}
\usepackage{txfonts}
\usepackage{mathrsfs}

\usepackage{enumitem}
\setlist[enumerate]{leftmargin=0pt, itemindent=17pt}
\setlist[itemize]{leftmargin=*}

\theoremstyle{plain}
\newtheorem{thm}{Theorem}[section]
\newtheorem{lem}[thm]{Lemma}
\newtheorem{cor}[thm]{Corollary}
\newtheorem{prop}[thm]{Proposition}

\theoremstyle{definition}
\newtheorem{dfn}[thm]{Definition}
\newtheorem{nttn}[thm]{Notation}

\theoremstyle{remark}
\newtheorem{rem}[thm]{Remark}

\begin{document}


\title{Jiang-Su Algebra as a Fra\"{i}ss\'{e} Limit}
\author[S.~Masumoto]{Shuhei MASUMOTO}
\address[S.~Masumoto]{Graduate~School~of~Mathematical~Sciences, the~University~of~Tokyo}
\email{masumoto@ms.u-tokyo.ac.jp}




\begin{abstract}
In this paper, we give a self-contained and quite elementary proof that 
the class of all dimension drop algebras together with their distinguished faithful traces 
forms a Fra\"{i}ss\'{e} class with the Jiang-Su algebra as its limit.  
We also show that the UHF algebras can be realized as Fra\"{i}ss\'{e} limits of 
classes of C*-algebras of matrix-valued continuous functions on $[0, 1]$ with faithful traces.  
\end{abstract}


\maketitle


\section{Introduction}


The Fra\"{i}ss\'{e} theory was originally invented by Roland Fra\"{i}ss\'{e}
in \cite{fraisse54:_extension_relations}, where a bijective correspondence between 
countable ultra-homogeneous structures and classes with certain properties 
of finitely generated structures is established.  
The classes and corresponding ultra-homogeneous structures in question 
are called \emph{Fra\"{i}ss\'{e} classes} and 
\emph{Fra\"{i}ss\'{e} limits} of the classes respectively.  


This theory has been, among the rest, a target of generalization to 
the setting of metric structures.  
For example, a general theory was developed in \cite{schoretsanitis07:_fraisse_theory}, 
including connections with bounded continuous logic.  
In \cite{yaacov15:_fraisse_limits}, Ita\"{i} Ben Yaacov concisely gave 
a self-contained presentation of a general theory, using a bright idea of approximate isomorphisms.  


These attempts at generalization ended up successfully, 
and a number of metric structures are recognized as Fra\"{i}ss\'{e} limits.  
Ita\"{i} Ben Yaacov~\cite{yaacov15:_fraisse_limits} pointed out that 
the Urysohn universal space, the separable infinite dimensional Hilbert space, 
and the atomless standard probability space are examples of Fra\"{i}ss\'{e} limits 
corresponding to suitable classes, and reconstructed discussion 
in \cite{kubis13:_proof_uniqueness} where the Gurarij space had been implicitly shown 
to be a Fra\"{i}ss\'{e} limit of the class of all finite dimensional Banach spaces.  
The latter result was quantized by Martino Lupini~\cite{lupini14:_uniqueness_universality}: 
it was shown that the noncommutative Gurarij space is the Fra\"{i}ss\'{e} limit of 
the class of all finite dimensional $1$-exact operator spaces.  


Among those instances are operator algebras.  In \cite{eagle15:_fraisse_limits}, 
a more generalized version of Fra\"{i}ss\'{e} theory for metric structures was presented, 
where the axioms of Fra\"{i}ss\'{e} class were relaxed, 
and so the bijective correspondence established in the original theory no longer holds 
and the limit structures would have less homogeneity, 
though it is still powerful as a construction method.  
Using this version, the authors of the paper succeeded in realizing 
a family of AF algebras including the UHF algebras, the hyperfinite II$_1$ factor 
and the Jiang-Su algebra as (generalized) Fra\"{i}ss\'{e} limits of 
a class of finite dimensional C*-algebras with distinguished traces, 
the class of finite dimensional factors and 
the class of dimension drop algebras with distinguished traces respectively.  


The Jiang-Su algebra was first constructed by Jiang and Su in \cite{jiang99:_simple_unital} 
as the unique simple monotracial C*-algebra among 
inductive limits of prime dimension drop algebras, 
which is KK-equivalent to the complex numbers $\mathbb{C}$.  
One of the most important properties of this algebra is that it is strongly self-absorbing, 
because of which it plays a key role in the Elliott's classification program of 
separable nuclear C*-algebras via K-theoretic invariants \cite{elliott08:_regularity_properties}.  
As is pointed out at the last section of \cite{eagle15:_fraisse_limits}, 
the proof that the Jiang-Su algebra satisfies this property is nontrivial, 
and there is a reasonable prospect that Fra\"{i}ss\'{e} theoretic view of this algebra 
will give a shortcut.  
However, the proof given in \cite{eagle15:_fraisse_limits} of the fact 
that the Jiang-Su algebra is a Fra\"{i}ss\'{e} limit was still
``a bit unsatisfactory'' in the authors' phrase, 
as it used the existence of the Jiang-Su algebra itself and 
relied heavily on Robert's theorem 
(see \cite[Remark 4.8 and Problem 7.2]{eagle15:_fraisse_limits}).  


In this paper, we prove that the collection of all the dimension drop algebras together with 
their distinguished faithful traces forms a Fra\"{i}ss\'{e} class.  
The importance lies in that this proof is self-contained and quite elementary; 
in particular, it depends on neither the existence of the Jiang-Su algebra nor 
Robert's theorem, so that it can be considered as a solution to 
\cite[Problem 7.2]{eagle15:_fraisse_limits}.    
Also, we show that the UHF algebras are realized as a Fra\"{i}ss\'{e} limit of 
a class of C*-algebras of matrix-valued continuous functions on 
the interval $[0, 1]$ together with their distinguished faithful traces.  
Since this class differs from the one used in \cite{eagle15:_fraisse_limits}, 
this result implies a different homogeneity of the UHF algebras.  


The paper consists of four sections.  
In the next section, we briefly introduce a version of Fra\"{i}ss\'{e} theory 
for metric structures, which is essentially the same as the one in \cite{eagle15:_fraisse_limits}.  
The third section contains the result on the UHF algebras.  
The argument included in this section is the basis of the fourth section, where 
the dimension drop algebras and the Jiang-Su algebra are dealt with.


\section{Fra\"{i}ss\'{e} Theory for Metric Structures}
\label{sec:_fraisse_theory}


In this section, we present a general theory of Fra\"{i}ss\'{e} limits 
in the context of metric structures, which is almost the same as the one 
in \cite[Section 2]{eagle15:_fraisse_limits}.  
The facts stated here are slight generalization of those of \cite{yaacov15:_fraisse_limits}, 
and can be proved with trivial modification.  


\begin{dfn}\label{dfn:_language_and_structure}
A \emph{language} $L$ consists of \emph{predicate symbols} and \emph{function symbols}.  
To each symbol in $L$ is associated a natural number called its \emph{arity}.  
We assume that $L$ contains a binary predicate symbol $d$.  


An \emph{$L$-structure} is a complete metric space $M$ together with an \emph{interpretation} of 
symbols of $L$: 
\begin{itemize}
	\item to each $n$-ary predicate symbol $P$ is assigned a continuous map 
		$P^M \colon M^n \to \mathbb{R}$, where the distinguished binary predicate symbol $d$ 
		corresponds to the distance function; and 
	\item to each $n$-ary function symbol $f$ is assigned a continuous map 
		$f^M \colon M^n \to M$.  
\end{itemize}
An \emph{embedding} of an $L$-structure $M$ into another $L$-structure $N$ is 
a map $\varphi$ such that 
\[
	f^N(\varphi(a_1), \ldots, \varphi(a_n)) = \varphi(f^M(a_1, \ldots, a_n))
\]
and 
\[
	P^N(\varphi(a_1), \ldots, \varphi(a_n)) = P^M(a_1, \ldots, a_n)
\]
hold for any function symbol $f$, any predicate symbol $P$ and any elements $a_1, \ldots, a_n$ in $M$.  
\end{dfn}


In this paper, we focus on unital C*-algebras with distinguished traces.  
We assume that $L$ consists of the binary predicate symbol $d$, 
an unary predicate symbol $\mathrm{tr}$, 
binary function symbols ${}+{}$ and ${}\cdot{}$, 
an unary function symbols $\lambda$ for each $\lambda \in \mathbb{C}$ 
which should be interpreted as multiplication by $\lambda$, 
an unary function symbol $*$, 
and $0$-ary function symbols $0$ and $1$.  
Then every unital C*-algebra with trace is understood as an $L$-structure in the canonical manner.  
Note that an embedding in the sense of Definition~\ref{dfn:_language_and_structure} is 
an injective \emph{trace-preserving} $*$-homomorphism in this case, 
which we shall call a \emph{morphism} in the sequel.  


\begin{rem}
The definition of languages and metric structures varies by paper 
(see \cite[Remark 2.2]{yaacov15:_fraisse_limits}), 
and the one we adopted here is the same as \cite[Definition 2.1]{yaacov15:_fraisse_limits}.  
Some variants such as \cite[Definition 2.1]{eagle15:_fraisse_limits} require 
that all the maps which appear should be bounded or uniformly continuous, 
in which case the language carries additional informations.  
A C*-algebra is seemingly not an instance of a metric structure in these cases, 
because it is apparently unbounded and the multiplication is not uniformly continuous.  
Indeed, this can be easily overcome by using the unit ball as its representative, 
as in \cite{eagle15:_fraisse_limits}.  
Anyway, the results of Fra\"{i}ss\'{e} theory in both perspectives can be easily translated to each other, 
so we are in the same line as \cite{eagle15:_fraisse_limits}.  
\end{rem}


\begin{dfn}
A class $\mathscr{K}$ of $L$-structures is said to satisfy
\begin{itemize}
	\item the \emph{joint embedding property (JEP)} 
		if for any $A, B \in \mathscr{K}$ there exists $C \in \mathscr{K}$ 
		such that both $A$ and $B$ can be embedded in $C$.  
	\item the \emph{near amalgamation property (NAP)} 
		if for any $A, B_1, B_2 \in \mathscr{K}$, 
		any embeddings $\varphi_i \colon A \to B_i$, any finite subset 
		$G \subseteq A$ and any $\varepsilon > 0$, 
		there exist embeddings $\psi_i$ of $B_i$ into some $C \in \mathscr{K}$ such that 
		$d(\psi_1 \circ \varphi_1 (a), \, \psi_2 \circ \varphi_2(a))$ is less than $\varepsilon$ for all $a \in G$.  
\end{itemize}
\end{dfn}


An $L$-structure $A$ is said to be \emph{finitely generated} 
if there exists a tuple $\vec{a} = (a_1, \ldots, a_n) \in A^n$ such that 
the smallest substructure of $A$ containing all $a_1, \ldots, a_n$ is $A$, 
for some $n \in \mathbb{N}$.  
(Note that we assumed $L$-structures to be necessarily complete.)
As we focus on unital C*-algebras with distinguished traces, 
this definition coincides with the usual one, that is, 
a (unital) C*-algebra (with its distinguished trace) 
is finitely generated if there exists a finite subset 
such that its closure by addition, multiplication, scalar multiplication and 
$*$-operation is dense in the whole C*-algebra.  


Let $\mathscr{K}$ be a class of finitely generated $L$-structures.  
For each $n \in \mathbb{N}$, we denote by $\mathscr{K}_n$ the class of 
all the pairs $(A, \vec{a})$, where $A$ is a member of $\mathscr{K}$ and $\vec{a} \in A^n$ is a generator of $A$.  
If $\mathscr{K}$ satisfies JEP and NAP, then we can define a pseudometric $d^\mathscr{K}$
on $\mathscr{K}_n$ by 
\[
	d^\mathscr{K}\bigl((A, \vec{a}), \, (B, \vec{b}) \bigr) := \inf  \max_i d(f(a_i), \, g(b_i)), 
\]
where $\vec{a} = (a_1, \ldots, a_n), \vec{b} = (b_1, \ldots, b_n)$ and 
the infimum is taken over all the embeddings $f, g$ of $A, B$ into some $C$ in $\mathscr{K}$.  


\begin{dfn}
A class $\mathscr{K}$ of finitely generated $L$-structures with JEP and NAP is said to satisfy 
\begin{itemize}
	\item the \emph{weak Polish property (WPP)} 
		if $\mathscr{K}_n$ is separable with respect to the pseudometric $d^\mathscr{K}$ for all $n$.  
	\item the \emph{Couchy Continuity Property (CCP)} if 
		\begin{enumerate}
			\item for any $n$-ary predicate symbol $P$, the map 
			$\bigl(A, (\vec{a}, \vec{b})\bigr) \mapsto P^A(\vec{a})$ 
			from $\mathscr{K}_{n+m}$ into $\mathbb{R}$ 
			sends Cauchy sequences to Cauchy sequences; and 
			\item for any $n$-ary function symbol $f$, the map 
			$\bigl(A, (\vec{a}, \vec{b})\bigr) \mapsto 
			\bigl(A, (\vec{a}, \vec{b}, f^A(\vec{a}))\bigr)$ from $\mathscr{K}_{n+m}$ into 
			$\mathscr{K}_{n+m+1}$ sends Cauchy sequences to Cauchy sequences.  
		\end{enumerate}
\end{itemize}
\end{dfn}


\begin{rem}\label{rem:_on_ccp}
CCP implies that $d^\mathscr{K}\bigl((A, \vec{a}), (B, \vec{b})\bigr) = 0$ holds if and only if 
there is an isomorphism between $A$ and $B$ sending $\vec{a}$ to $\vec{b}$ 
(\cite[Remark 2.13 (i)]{yaacov15:_fraisse_limits}).  
Note that if $\mathscr{K}$ is a class of finitely generated 
unital C*-algebras with traces and if it satisfies JEP and NAP, 
then it also satisfies CCP automatically, 
because all the relevant functions are $1$-Lipschitz on the unit ball.  
\end{rem}


\begin{dfn}
A class $\mathscr{K}$ of finitely generated $L$-structures is called a \emph{Fra\"{i}ss\'{e} class} if 
it satisfies JEP, NAP, WPP and CCP.  A \emph{Fra\"{i}ss\'{e} limit} of a Fra\"{i}ss\'{e} class $\mathscr{K}$ is 
a separable $L$-structure $M$ which is
\begin{enumerate}
	\item a \emph{$\mathscr{K}$-structure}: 
		for any finite subset $F$ of $M$ and any $\varepsilon>0$, 
		there exists an embedding $\varphi$ of a member of $\mathscr{K}$ such that 
		the $\varepsilon$-neighborhood of the image of $\varphi$ includes $F$.  
	\item \emph{$\mathscr{K}$-universal}: 
		every member of $\mathscr{K}$ can be embedded into $M$.  
	\item \emph{approximately $\mathscr{K}$-homogeneous}: 
		if $A$ is a member of $\mathscr{K}$ and $a_1, \ldots, a_n$ are elements of $A$, 
		then for any embeddings $\varphi, \psi$ of $A$ into $M$ and any $\varepsilon > 0$, 
		there exists an automorphism $\alpha$ of $M$ with 
		$d(\alpha \circ \varphi(a_i), \, \psi(a_i)) < \varepsilon$ for $i = 1, \ldots, n$.  
\end{enumerate}
\end{dfn}


The definition here is more relaxed than that of \cite{yaacov15:_fraisse_limits} and 
close to \cite[Definition 2.6]{eagle15:_fraisse_limits}: 
our Fra\"{i}ss\'{e} class is incomplete and lacks the hereditary property 
(see \cite[Definitions 2.5 (ii) and 2.12]{yaacov15:_fraisse_limits}).  
Consequently, we cannot establish a bijective correspondence 
between Fra\"{i}ss\'{e} classes and separable structures with homogeneity, 
which is a part of the main result of Fra\"{i}ss\'{e} theory.  
The following theorem summarizes what remains in our framework.  


\begin{thm}\label{thm:_main_result_of_fraiise_theory}
Every Fra\"{i}ss\'{e} class $\mathscr{K}$ admits a unique limit.  
Moreover, for any $L$-structure $A_0$ in $\mathscr{K}$, there exists a sequence of embeddings 
$A_0 \xrightarrow{\varphi_0} A_1 \xrightarrow{\varphi_1} A_2 \xrightarrow{\varphi_2} \cdots$ such that 
$A_i$ belongs to $\mathscr{K}$ for all $i$ and its inductive limit coincides with 
the Fra\"{i}ss\'{e} limit of $\mathscr{K}$.  
\end{thm}


\section{UHF Algebras}
\label{sec:_uhf_algebras}


A \emph{supernatural number} is a formal product 
\[
	\nu = \prod_{p \text{: prime}} p^{n_p}, 
\]
where $n_p$ is either a non-negative integer or $\infty$ for each $p$ such that $\sum_p n_p = \infty$.  
In \cite[Theorem 1.12]{glimm60:_certain_class} it was proved that to each UHF algebra is associated 
a supernatural number as its complete invariant.  
Now, given a supernatural number $\nu$, 
we denote by $\mathbb{N}(\nu)$ the set of all natural numbers which formally divides $\nu$, 
and by $\mathscr{K}(\nu)$ the class of 
all the pairs $\langle C[0, 1] \otimes \mathbb{M}_n, \tau \rangle$, 
where $n$ is in $\mathbb{N}(\nu)$ and $\tau$ is a faithful trace 
on the C*-algebra $C[0, 1] \otimes \mathbb{M}_n$.  
Our goal in this section is to show that $\mathscr{K}(\nu)$ is a Fra\"{i}ss\'{e} class the limit of which is 
the UHF algebra with $\nu$ as its associated supernatural number.  


First, note that $C[0, 1] \otimes \mathbb{M}_n$ is canonically isomorphic to $C([0, 1], \ \mathbb{M}_n)$, 
the C*-algebra of all the continuous $\mathbb{M}_n$-valued functions on the interval $[0, 1]$.  
In the sequel, we shall denote this C*-algebra by $\mathcal{A}_n$ for simplicity.  
Next, let $\tau$ be a probability Radon measure on $[0, 1]$, 
which is identified with a state on $C[0, 1]$ by integration.  
Then $\tau \otimes \mathrm{tr}$ is clearly a trace on $\mathcal{A}_n$, 
where $\mathrm{tr}$ is the unique normalized trace on $\mathbb{M}_n$.  
It is easy to see that every trace on $\mathcal{A}_n$ is of this form, 
so a probability Radon measure on $[0, 1]$ may also be identified with a trace on $\mathcal{A}_n$.   
In the sequel, we simply write $\tau$ instead of $\tau \otimes \mathrm{tr}$ 
and use adjectives for measures and traces in common.  
For example, a measure is said to be faithful if its corresponding trace is faithful.  
Also, all the measures are assumed to be probability Radon measures 
so that they always correspond to traces.  


A measure is said to be \emph{diffuse} or \emph{atomless} if any measurable set of nonzero measure is 
parted into two measurable sets of nonzero measure.  
The following is often used in the sequel without referring.  


\begin{lem}\label{lem:_measures_and_monotone_functions}
Let $\sigma, \tau$ be faithful measures.  
If $\sigma$ is diffuse, then there exists a unique non-decreasing continuous function 
$\beta$ from $[0, 1]$ onto $[0, 1]$ with $\beta_*(\sigma) = \tau$.  
Moreover, $\tau$ is diffuse if and only if $\beta$ is a homeomorphism.  
\end{lem}


\begin{proof}
We first assume $\sigma$ is equal to the Lebesgue measure $\lambda$ and set $\alpha(t) := \tau([0, t))$.  
Note that $\alpha$ is a strictly increasing upper semi-continuous function from $[0, 1]$ into $[0, 1]$.  
Let $\beta$ be the unique non-decreasing function extending $\alpha^{-1}$.  
Then 
\[
	\beta_*(\lambda)([0, t)) = \lambda\bigl(\beta^{-1}([0, t))\bigr) 
		= \lambda\bigl([0, \alpha(t))\bigr) = \alpha(t) = \tau([0, t)), 
\]
so $\beta_*(\lambda)$ is equal to $\tau$.  Also, if $\tau$ is diffuse, then $\alpha$ is continuous, 
whence $\beta = \alpha^{-1}$ is  a homeomorphism.  


For the general case, let $\beta_\sigma, \beta_\tau$ be such that 
$(\beta_\sigma)_*(\lambda) = \sigma$ and $(\beta_\tau)_*(\lambda) = \tau$.  
Then $\beta := \beta_\tau \circ \beta_\sigma^{-1}$ satisfies $\beta_*(\sigma) = \tau$, 
which completes the proof.  
\end{proof}


The next propositions are immediate corollaries of the preceding lemma.  
Recall that a morphism between elements of $\mathscr{K}(\nu)$ is 
an injective unital trace-preserving $*$-homomorphism.  


\begin{prop}\label{prop:_embeddability_lemma}
Let $\tau$ be a faithful trace on $\mathcal{A}_n$.  
Then for any  faithful diffuse trace $\sigma$ on $\mathcal{A}_n$, 
there is a morphism $\varphi \colon \langle \mathcal{A}_n, \tau \rangle \to \langle \mathcal{A}_n, \sigma \rangle$.   
\end{prop}


\begin{proof}
Let $\beta$ be the non-decreasing continuous function as 
in Lemma~\ref{lem:_measures_and_monotone_functions}.  
Then $\varphi := \beta^*$ is the desired morphism.  
\end{proof}


\begin{prop}\label{prop:_jep_for_matrix-valued_function_algebras}
The class $\mathscr{K}(\nu)$ satisfies JEP.  
\end{prop}


\begin{proof}
Suppose $n_1, n_2$ are in $\mathbb{N}(\nu)$ and put $n := \mathrm{gcd}(n_1, n_2)$.  
Then $n$ is also in $\mathbb{N}(\nu)$.  
Also, $\langle \mathcal{A}_{n_i}, \lambda \rangle$ is clearly embeddable into 
$\langle \mathcal{A}_n, \lambda \rangle$ by amplification.  
This fact together with Proposition~\ref{prop:_embeddability_lemma} 
implies that $\mathscr{K}(\nu)$ satisfies JEP.  
\end{proof}


Next, we shall show that the class $\mathscr{K}(\nu)$ satisfies NAP.  
For this, we begin with proving that all the morphisms of $\mathscr{K}(\nu)$ are 
approximately diagonalizable.  


\begin{dfn}
A morphism $\varphi \colon \langle \mathcal{A}_n, \tau \rangle \to \langle \mathcal{A}_m, \sigma \rangle$ is 
said to be \emph{diagonalizable} if there are a unitary $u \in \mathcal{A}_m$ and 
continuous maps $\xi_1, \ldots, \xi_k \colon [0, 1] \to [0, 1]$ such that 
\begin{equation}\label{eq:diag}
	\varphi(f) = u \,
		\begin{pmatrix}
			f \circ \xi_1 & & 0 \\
			 & \ddots & \\
			0 & & f \circ \xi_k
		\end{pmatrix}
		\, u^*
\end{equation}
for all $f \in \mathcal{A}_n$.  
\end{dfn}


In this paper, we shall call Eq.~(\ref{eq:diag}) a \emph{diagonal expression} of $\varphi$, and 
$u$ and $\xi_1, \ldots, \xi_k$ its \emph{associated unitary and maps}.  
Note that the union of the images of  the maps associated to a diagonal expression is equal to $[0, 1]$, 
as morphisms are necessarily faithful.  
Also, compositions of diagonalizable morphisms are again diagonalizable.  


\begin{prop}\label{prop:_diagonalizability_1}
Let $\varphi \colon \langle \mathcal{A}_n, \tau \rangle \to \langle \mathcal{A}_m, \sigma \rangle$ be 
a morphism.  Then for any finite subset $G \subseteq \mathcal{A}_n$ and any $\varepsilon > 0$, 
there exists a diagonalizable morphism $\psi \colon \langle \mathcal{A}_n, \tau \rangle \to 
\langle \mathcal{A}_m, \sigma \rangle$ with $\|\varphi(g) - \psi(g)\| < \varepsilon$ for all $g \in G$.  
Moreover, we can take $\psi$ so that the maps $\xi_1, \ldots, \xi_k$ 
associated to a diagonal expression of $\psi$ satisfies $\xi_1 \leq \cdots \leq \xi_k$.  
\end{prop}


\begin{proof}
For $t \in [0, 1]$, let $\mathrm{ev}_t \colon \mathcal{A}_m \to \mathbb{M}_m$ be 
the evaluation $*$-homomorphism.  
Then $\mathrm{ev}_t \circ \varphi$ is a unital $*$-homomorphism from $\mathcal{A}_n$ to 
the finite dimensional C*-algebra $\mathbb{M}_m$, so there exist a unitary $v_t \in \mathbb{M}_m$ 
and real numbers $s_1^t, \ldots, s_k^t \in [0, 1]$ such that the equation
\[
	\mathrm{ev}_t \circ \varphi (f) = \mathrm{Ad}(v_t) \bigl(\mathrm{diag}(f(s_1^t), \ldots, f(s_k^t)) \bigr)
\]
holds for all $f \in \mathcal{A}_n$.  Note that $\{\{ s_1^t, \ldots, s_k^t \}\}$ coincides with 
the spectra of $\mathrm{ev}_t \circ \varphi (\mathrm{id}_{[0, 1]} \otimes 1_{\mathbb{M}_n})$ 
as multisets.  
By continuity, if $t_1$ and $t_2$ are close to each other, then so are the spectrum of 
$\mathrm{ev}_{t_1} \circ \varphi (\mathrm{id}_{[0, 1]} \otimes 1_{\mathbb{M}_n})$ and 
$\mathrm{ev}_{t_2} \circ \varphi (\mathrm{id}_{[0, 1]} \otimes 1_{\mathbb{M}_n})$ 
with respect to the Hausdorff distance.  
Therefore, if we define 
\[
\begin{aligned}
	\xi_1(t) &:= \max \{\{s_1^t, \ldots, s_k^t\}\}; \\
	\xi_i(t) &:= \max \{\{s_1^t, \ldots, s_k^t\}\} 
		\setminus \{\{\xi_1(t), \ldots, \xi_{i-1}(t)\}\}, 
\end{aligned}
\]
then obviously $\xi_1, \ldots, \xi_k$ are continuous functions 
from $[0, 1]$ into $[0, 1]$ satisfying $\xi_1 \leq \cdots \leq \xi_k$.  


Next, fix $t_0 \in [0, 1]$.  We claim that there exists $\delta(t_0) > 0$ with the following property: 
if $|t - t_0| < \delta(t_0)$, then there exists a unitary $w_{t_0} \in \mathbb{M}_m$ with 
$\|v_t - w_{t_0}\| < \varepsilon$ such that the equation 
\[
	\mathrm{ev}_{t_0} \circ \varphi (f) = \mathrm{Ad}(w_{t_0}) 
		\bigl(\mathrm{diag}(f(s_1^{t_0}), \ldots, f(s_k^{t_0})) \bigr)
\]
holds for all $f \in \mathcal{A}_n$.  
To see this, let $s_1, \ldots, s_l$ be \emph{distinct} eigenvalues of 
$\mathrm{ev}_{t_0} \circ \varphi(\mathrm{id}_{[0, 1]} \otimes 1_{\mathbb{M}_n})$ and 
take mutually orthogonal non-negative continuous functions $f_1, \ldots, f_l$ such that 
$f_i$ is constantly equal to $1$ on some neighborhood of $s_i$ for each $i$.  
Note that if $\{e_{p, q}\}$ is a matrix unit of $\mathbb{M}_n$, 
then $\{\mathrm{ev}_{t_0} \circ \varphi(f_i \otimes e_{p, q})\}_{i, p, q}$ forms a matrix unit of 
$\mathrm{Im}(\mathrm{ev}_{t_0} \circ \varphi)$, 
and if $t$ is sufficiently close to $t_0$, then $\{\mathrm{ev}_t \circ \varphi(f_i \otimes e_{p, q})\}_{i, p, q}$ 
is a matrix unit of a subalgebra of $\mathrm{Im}(\mathrm{ev}_t \circ \varphi)$ which is close to 
$\{\mathrm{ev}_{t_0} \circ \varphi(f_i \otimes e_{p, q})\}_{i, p, q}$.  
Hence, as in the proof of \cite[Lemma III.3.2]{davidson96:_cstar_algebras}, 
we can find a unitary $w$ with $\|w-1\| < \varepsilon$ such that 
\[
	w(\mathrm{ev}_{t_0} \circ \varphi (f_i \otimes e_{p, q}))w^* = 
		\mathrm{ev}_t \circ \varphi (f_i \otimes e_{p, q}), 
\]
and $w_{t_0} := v_t w$ has the desired property.  


Now take $\delta_0 > 0$ sufficiently small so that the inequalities 
\[
	\|g \circ \xi_i(s) - g \xi_i(t)\| < \varepsilon, \qquad 
	\|\mathrm{ev}_s \circ \varphi(g)  - \mathrm{ev}_t \circ \varphi(g)\| < \varepsilon
\]
hold for all $g \in G$ whenever $|s-t| < \varepsilon$, and consider an open covering 
\[
	\mathscr{U} := \bigl\{ U_\delta(t) \mid t \in [0, 1] \ \& \ \delta < \min \{\delta(t), \delta_0\} \bigr\}
\]
of $[0, 1]$, where $U_\delta(t)$ denotes the open ball of radius $\delta$ and center $t$.  
Since $[0, 1]$ is compact, there exists a finite subcovering, say $\{I_1, \ldots, I_r\}$.  
We denote the center of $I_j$ by $c_j$, and without loss of generality, 
we may assume $c_1 < \cdots < c_r$ and $I_j \cap I_{j+1} \neq \varnothing$ for all $j$.  
Take $\eta > 0$ and $b_j \in I_j \cap I_{j+1} \cap (c_j+\eta, c_{j+1}-\eta)$ for each $j$, 
and find a unitary $u \in \mathcal{A}_m$ such that
\begin{itemize}
	\item $u(b_j)$ is equal to $v_{b_j}$ for all $j$; 
	\item the image of $u$ on $[c_j+\eta, c_{j+1}-\eta]$ is included in the $\varepsilon$-ball of center $u(b_j)$; 
		and
	\item the image of $u$ on $[c_j-\eta, c_j + \eta]$ is included in the path-connected subset
		\[
			\Bigl\{ w ~\Bigm|~ \mathrm{ev}_{c_j} \circ \varphi (f) = \mathrm{Ad}(w) 
			\bigl( \mathrm{diag}(f \circ \xi_1(c_j), \ldots, f \circ \xi_k(c_j)) \bigr) \Bigr\}
		\]
		of unitaries, 
\end{itemize}
which is possible by the claim we proved in the previous paragraph.  


We shall set 
\[
	\psi(f) := \mathrm{Ad}(u) \bigl( \mathrm{diag}(f \circ \xi_1, \ldots, f \circ \xi_k) \bigr) 
\]
and show that this $\psi$ has the desired property.  
First, it is clear from the definition of $\xi_i$ that $\psi$ is trace-preserving.  
Now, suppose $t \in [c_j + \eta, c_{j+1} - \eta]$ and $g \in G$.  
Without loss of generality, we may assume that the norm of $g$ is less than $1$.  
Then we have
\begin{align*}
	\mathrm{ev}_t \circ \psi(g) 
		&= \mathrm{Ad}(u(t)) \bigl( \mathrm{diag} (g \circ \xi_1(t), \ldots, g \circ \xi_k(t)) \bigr) \\
		&\sim_{3\varepsilon} \mathrm{Ad}(u(b_j)) 
		\bigl( \mathrm{diag} (g \circ \xi_1(b_j), \ldots, g \circ \xi_k(b_j)) \bigr) 
		= \mathrm{ev}_{b_j} \circ \varphi(g) \\
		&\sim_\varepsilon \mathrm{ev}_t \circ \varphi(g).  
\end{align*}
On the other hand, if $t \in [c_j - \eta, c_j+\eta]$, then 
\begin{align*}
	\mathrm{ev}_t \circ \psi(g) 
		&= \mathrm{Ad}(u(t)) \bigl( \mathrm{diag} (g \circ \xi_1(t), \ldots, g \circ \xi_k(t)) \bigr) \\
		&\sim_\varepsilon \mathrm{Ad}(u(t)) 
		\bigl( \mathrm{diag} (g \circ \xi_1(c_j), \ldots, g \circ \xi_k(c_j)) \bigr) 
		= \mathrm{ev}_{c_j} \circ \varphi(g) \\
		&\sim_\varepsilon \mathrm{ev}_t \circ \varphi(g).
\end{align*}
Consequently, it follows that $\|\varphi(g) - \psi(g)\| < 4\varepsilon$ for all $g \in G$, which completes the proof.  
\end{proof}


The following proposition is also an immediate corollary of 
Lemma~\ref{lem:_measures_and_monotone_functions}.  
However, as a preparation to the next section, 
we shall present a slightly ineffective proof.  


\begin{prop}\label{prop:_cutting-up_lamma}
Let $\tau, \sigma$ be faithful diffuse measures on $[0, 1]$.  
Then for any $n \in \mathbb{N}(\nu)$ and any $\varepsilon > 0$, 
there exist $m \in \mathbb{N}(\nu)$ and a diagonalizable morphism 
$\varphi \colon \langle \mathcal{A}_n, \tau \rangle \to \langle \mathcal{A}_m, \sigma \rangle$ 
such that the images of the maps associated to a diagonal expression of $\varphi$ 
have diameters less than $\varepsilon$.  
\end{prop}


\begin{proof}
Since $\tau$ is diffuse, there exists $\delta > 0$ such that $\tau([t_1, t_2]) < \delta$ implies $|t_2 - t_1| < 1/6$.  
Take $k \in \mathbb{N}$ so that $m_1 := nk$ is in $\mathbb{N}(\nu)$ and $1/k$ is smaller than $\delta$.  
Then there are $t_0 \in (1/2, 2/3)$ and $r \in \mathbb{N}$ with $\tau([0, t_0]) = r/k$.  
We set $\tau_1 := \frac{k}{r}\tau|_{[0, t_0]}$ and $\tau_2 := \frac{k}{k-r}\tau_{[t_0, 1]}$, 
so $\tau = \frac{r}{k}\tau_1 + \frac{k-r}{k}\tau_2$.  
By Lemma~\ref{lem:_measures_and_monotone_functions}, 
one can find increasing maps $\eta_1 \colon [0, 1] \to [0, t_0]$ and $\eta_2 \colon [0, 1] \to [t_0, 1]$ 
such that $\tau_i$ is equal to $(\eta_i)_*(\sigma)$ for $i = 1, 2$.  
We set 
\[
	\xi^1_j := \left\{
		\begin{array}{ll}
			\eta_1 & \text{if } j = 1, \ldots, r, \\
			\eta_2 & \text{if } j = r+1, \ldots, k,   
		\end{array}
		\right.
\]
and define $\varphi_1 \colon \mathcal{A}_n \to \mathcal{A}_{m_1}$ by 
\[
	\varphi_1(f) = \mathrm{diag}(f \circ \xi^1_1, \ldots, f \circ \xi^1_k).  
\]
Then it can be easily verified that $\varphi_1$ is a morphism from $\langle \mathcal{A}_n, \tau \rangle$ to 
$\langle \mathcal{A}_{m_1}, \sigma \rangle$, 
and that the images of the maps $\xi^1_1, \ldots, \xi^1_k$ are either $[0, t_0]$ or $[t_0, 1]$, 
so their diameters are less than $2/3$.  


Now take $d \in \mathbb{N}$ large enough so that $(2/3)^d$ is less than $\varepsilon$, 
and repeat the procedure above for $d$ times to obtain a sequence 
\[
	\langle \mathcal{A}_n, \tau \rangle \xrightarrow{\ \varphi_1 \ }
	\langle \mathcal{A}_{m_1}, \sigma \rangle \xrightarrow{\ \varphi_2 \ } \cdots 
	\xrightarrow{\varphi_{d-1}} \langle \mathcal{A}_{m_d}, \sigma \rangle.  
\]
Then $\varphi := \varphi_{d-1} \circ \cdots \circ \varphi_1$ has the desired property.  
\end{proof}


\begin{prop}\label{prop:_nap_for_matrix-valued_function_algebras}
The class $\mathscr{K}(\nu)$ satisfies NAP.  
\end{prop}


\begin{proof}
Let $\varphi_1, \varphi_2$ be morphisms from $\langle \mathcal{A}_{n_0}, \tau_0 \rangle$ 
into $\langle \mathcal{A}_{m'}, \sigma' \rangle$, $\langle \mathcal{A}_{m''}, \sigma'' \rangle$ 
respectively, and $G$ be a finite subset of $\mathcal{A}_{n_0}$.  
Our goal is to show that given $\varepsilon > 0$, we can find morphisms $\psi_1$ and $\psi_2$ 
from $\langle \mathcal{A}_{m'}, \sigma' \rangle$ and $\langle \mathcal{A}_{m''}, \sigma'' \rangle$ 
respectively into some $\langle \mathcal{A}_{n_2}, \tau_2 \rangle \in \mathscr{K}(\nu)$ with 
$\|\psi_1 \circ \varphi_1(g) - \psi_2 \circ \varphi_2(g)\| < \varepsilon$ for all $g \in G$.  
To see this, we may assume that $m' = m'' =: n_1$ and $\sigma' = \sigma'' =: \tau_1$, 
by Proposition~\ref{prop:_jep_for_matrix-valued_function_algebras}; 
and that both $\varphi_1$ and $\varphi_2$ are diagonalizable, 
by Proposition~\ref{prop:_diagonalizability_1}.  


Let $\zeta^i_1, \ldots, \zeta^i_l$ be the maps associated to a diagonal expression of $\varphi_i$.  
Take $\delta > 0$ so that $|s-t| < \delta$ implies $|g(s)-g(t)| < \varepsilon$ for any $g \in G$, 
and apply Proposition~\ref{prop:_cutting-up_lamma} to obtain a morphism 
$\rho$ from $\langle \mathcal{A}_{n_1}, \tau_1 \rangle$ into some $\langle \mathcal{A}_{n_2}, \tau_2 \rangle$ 
such that the images of the maps associated to a diagonal expression of 
$\rho \circ \tilde{\varphi}_i$ have diameters less than $\delta/3$ for each $i$.  
Then applying Proposition~\ref{prop:_jep_for_matrix-valued_function_algebras}, 
find a diagonalizable morphism $\Phi_i$ such that the inequaily
$\|\rho \circ \varphi_i(g) - \Phi_i(g)\| < \varepsilon$ holds for $g \in G$, and that
the maps $\xi^i_1, \ldots, \xi^i_k$ associated to a diagonal expression of $\Phi_i$ satisfies 
$\xi^i_1 \leq \cdots \leq \xi^i_k$.  
Recalling the proof of Proposition~\ref{prop:_diagonalizability_1}, 
one can easily check that the diameters of the images of $\xi^i_j$ is still less than $\delta/3$.  


We claim that the inequality $\|\xi^1_j - \xi^2_j\| < \delta$ holds for all $j$.  
Suppose on the contrary that $\xi^1_j(t) \geq \xi^2_j(t) + \delta$ at some point $t \in [0, 1]$, 
and set $c := \max \xi^2_j$, $d := \min \xi^1_{j+1}$.  
(If $j$ is equal to $k$, then set $d := 1$ instead.)  
Then it follows that 
\begin{itemize}
	\item the image of $\xi^2_l$ is included in $[0, c]$ if $1 \leq l \leq j$; and 
	\item if the image of $\xi^1_l$ intersects with $[0, d)$, then $l$ is less than or equal to $j$.  
\end{itemize}
Since $d$ is larger than $c$ by at least $\delta/3$, and since $\Phi_i$ is trace-preserving, we have 
\[
	j = \sum_{l=1}^j \tau_2\bigl((\xi^2_l)^{-1}[0, c]\bigr) 
		\leq n_2 \tau_0\bigl( [0, c] \bigr)
		< n_2 \tau_0\bigl( [0, d) \bigr)
		\leq \sum_{l=1}^j \tau_1\bigl((\xi^1_l)^{-1}[0, 1]\bigr) = j, 
\]
which is a contradiction.  
Therefore, $\|\xi^1_j - \xi^2_j\|$ must be smaller than $\delta$, as desired.  


Now, let $u_i$ be a unitary such that the equality
\[
	\Phi_i(f) = \mathrm{Ad}(u_i)\bigl( \mathrm{diag}(f \circ \xi^i_1, \ldots, f \circ \xi^i_k) \bigr)
\]
holds for all $f \in \mathcal{A}_{n_0}$, 
and put $\psi_1 := \rho$ and $\psi_2 := \mathrm{Ad}(u_1 u_2^*) \circ \rho$.  
Then for $g \in G$, we have
\begin{align*}
	\psi_2 \circ \varphi_2(g) &= \mathrm{Ad}(u_1 u_2^*) \circ \rho \circ \varphi_2(g) \\
		&\sim_\varepsilon \mathrm{Ad}(u_1 u_2^*) \circ \Phi_2(g) \\
		&= \mathrm{Ad}(u_1)\bigl( \mathrm{diag}(g \circ \xi^2_1, \ldots, g \circ \xi^2_k) \bigr) \\
		&\sim_\varepsilon \mathrm{Ad}(u_1)\bigl( \mathrm{diag}(g \circ \xi^1_1, \ldots, g \circ \xi^1_k) \bigr) \\
		&= \Phi_1(g) \\
		&\sim_\varepsilon \psi_1 \circ \varphi_1(g), 
\end{align*}
which completes the proof.  
\end{proof}


\begin{thm}\label{thm:_knu_is_a_fraisse_class}
The class $\mathscr{K}(\nu)$ is a Fra\"{i}ss\'{e} class.  
\end{thm}


\begin{proof}
We have already shown that $\mathscr{K}(\nu)$ satisfies JEP and NAP in 
Propositions~\ref{prop:_jep_for_matrix-valued_function_algebras} and 
\ref{prop:_nap_for_matrix-valued_function_algebras}.  
Also, it can be easily verified form the proof of 
Proposition~\ref{prop:_embeddability_lemma} that $\mathscr{K}(\nu)$ satisfies WPP.  
Since $\mathscr{K}(\nu)$ automatically satisfies CCP, as is noted in Remark~\ref{rem:_on_ccp}, 
it follows that $\mathscr{K}(\nu)$ is a Fra\"{i}ss\'{e} class.  
\end{proof}


We close this section by showing that the Fra\"{i}ss\'{e} limit of $\mathscr{K}(\nu)$ is the unique UHF algebra 
$\mathbb{M}_\nu$ corresponding to the supernatural number $\nu$.  
The following lemma will be needed for the proof.  


\begin{lem}\label{lem:_perturbation_lemma}
Let $\langle \mathcal{A}_n, \tau \rangle$ be a member of $\mathscr{K}(\nu)$.  
Then for any finite subset $F \subseteq \mathcal{A}_n$ and any $\varepsilon > 0$, 
there exist a morphism from $\langle \mathcal{A}_n, \tau \rangle$ into some 
$\langle \mathcal{A}_m, \sigma \rangle \in \mathscr{K}(\nu)$ and a finite dimensional 
C*-subalgebra $\mathcal{B} \subseteq \mathcal{A}_m$ such that 
the image $\varphi[F]$ is included in the $\varepsilon$-neighborhood of $\mathcal{B}$.  
\end{lem}


\begin{proof}
We may assume $F = \{ \mathrm{id}_{[0, 1]} \otimes 1_{\mathbb{M}_n} \} \cup 
\{ 1_{C[0, 1]} \otimes e_{i, j} \mid i, j = 1, \ldots, n \}$ where $\{e_{i, j}\}$ is the standard matrix unit of 
$\mathbb{M}_n$, because this set generates $\mathcal{A}_n$.  
Also, we may assume that $\tau$ is diffuse by Proposition~\ref{prop:_embeddability_lemma}.  
Now, let $\varphi$ be as in Proposition~\ref{prop:_cutting-up_lamma}.  
Then 
\begin{align*}
	\varphi(\mathrm{id}_{[0, 1]} \otimes 1_{\mathbb{M}_n}) 
		&= \mathrm{diag}(\xi_1, \ldots, \xi_k) \\
		&\sim_\varepsilon 1_{C[0, 1]} \otimes \mathrm{diag}(\xi_1(0), \ldots, \xi_k(0)), 
\end{align*}
so $\varphi[F]$ is included in the $\varepsilon$-neighborhood of the
finite dimensional C*-subalgebra $1_{C[0, 1]} \otimes \mathbb{M}_m$, as desired.  
\end{proof}


\begin{thm}
The Fra\"{i}ss\'{e} limit of $\mathscr{K}(\nu)$ is $\langle \mathbb{M}_\nu, \mathrm{tr} \rangle$, 
where $\mathrm{tr}$ is the unique trace on $\mathbb{M}_\nu$.    
\end{thm}


\begin{proof}
Let $\langle \mathcal{A}, \theta \rangle$ be the Fra\"{i}ss\'{e} limit of $\mathscr{K}$.  
By $\mathscr{K}$-universality and Theorem~\ref{thm:_main_result_of_fraiise_theory}, 
it is clear that $\mathcal{A}$ and $\mathbb{M}_\nu$ have the same K-theory.  
Therefore, it suffices to show that $\mathcal{A}$ is an AF algebra.  
For this, let $F$ be a subset of $\mathcal{A}$.  
Then given $\varepsilon > 0$, we can find a morphism $\varphi$ of some 
$\langle \mathcal{A}_n, \tau \rangle \in \mathscr{K}(\nu)$ into $\langle \mathcal{A}, \theta \rangle$ 
and a finite subset $F' \subseteq \mathcal{A}_n$ such that $F$ is included in 
the $\varepsilon$-neighborhood of $\varphi[F']$.  
On the other hand, by Lemma~\ref{lem:_perturbation_lemma}, 
there is a morphism $\psi$ from $\langle \mathcal{A}_n, \tau \rangle$ into some 
$\langle \mathcal{A}_m, \sigma \rangle \in \mathscr{K}(\nu)$ 
such that $\psi[F']$ is included in the $\varepsilon$-neighborhood of 
a finite dimensional C*-subalgebra of $\mathcal{A}_m$.  
Since $\langle \mathcal{A}, \theta \rangle$ is $\mathscr{K}$-universal and 
approximately $\mathscr{K}$-homogeneous, 
there is a morphism $\iota \colon \langle \mathcal{A}_m, \sigma \rangle \to \langle \mathcal{A}, \theta \rangle$ 
such that $d(\varphi(f), \iota \circ \psi(f))$ is less than $\varepsilon$ for all $f \in F'$.  
It follows that $F$ is included in the $3\varepsilon$-neighborhood of a finite dimensional C*-subalgebra 
of $\mathcal{A}$, so by \cite[Theorem III.3.4]{davidson96:_cstar_algebras}, 
$\mathcal{A}$ is an AF algebra, which completes the proof.  
\end{proof}

\section{The Jiang-Su Algebra}


Let $p, q$ be natural numbers.  
We shall begin with the well-known observation that 
if $\{e_{ij}\}_{i, j = 1}$ and $\{f_{kl}\}_{k, l = 1}$ are 
the standard matrix units of $\mathbb{M}_p$ and $\mathbb{M}_q$ respectively, 
then $\{e_{ij} \otimes f_{kl}\}_{i, j, k, l}$ is a matrix unit of 
$\mathbb{M}_p \otimes \mathbb{M}_q$, so $\mathbb{M}_p \otimes \mathbb{M}_q$ is 
canonically identified with $\mathbb{M}_{pq}$.  
Now, the \emph{dimension drop algebra} $\mathcal{Z}_{p, q}$ is defined by 
\[
	\mathcal{Z}_{p, q} := \{ f \in \mathcal{A}_{pq} \mid 
		f(0) \in \mathbb{M}_p \otimes 1_{\mathbb{M}_q} 
		\ \& \ f(1) \in 1_{\mathbb{M}_p} \otimes \mathbb{M}_q \}, 
\]
where we took over the notation $\mathcal{A}_n = C([0, 1], \mathbb{M}_n)$ 
from Section~\ref{sec:_uhf_algebras}.  
It is said to be \emph{prime} if $p$ and $q$ are co-prime.  
We denote by $\mathscr{K}$ the class of all pairs $\langle \mathcal{Z}_{p, q}, \tau \rangle$, 
where $\mathcal{Z}_{p, q}$ is a prime dimension drop algebra and $\tau$ is a faithful trace on it.  


In \cite{jiang99:_simple_unital}, Jiang and Su constructed the Jiang-Su algebra 
as an inductive limit of prime dimension drop algebras, 
and proved that it is the unique monotracial simple C*-algebra among such inductive limits.  
Our goal here is to show that the Jiang-Su algebra together with its unique trace is 
the Fra\"{i}ss\'{e} limit of the class $\mathscr{K}$.  
The direction of the proof is the same as that of Section~\ref{sec:_uhf_algebras}, 
but we need some additional observations because of the pinching condition.  


\begin{nttn}\label{nttn:_fs_and_nst}
Let $\xi = (\xi_1, \ldots, \xi_k)$ be a tuple of functions from $[0, 1]$ to $[0, 1]$.  
For $s = 0, 1$, we set $F_s(\xi) = \{\xi_1(s), \ldots, \xi_k(s)\}$.  
Also, for $t \in F_s(\xi)$, we denote by $n_s^t(\xi)$ the number of $i$ with $\xi_i(s) = t$.  
If the family $\xi$ under consideration is apparent from context, 
then $F_s(\xi)$ and $n_s^t(\xi)$ are simply written as $F_s$ and $n_s^t$ respectively.  
\end{nttn}


\begin{lem}\label{lem:_arithmetic_condition}
Let $\varphi \colon \mathcal{Z}_{p, q} \to \mathcal{A}_{p'q'}$ be 
a $*$-homomorphism of the form 
\[
	\varphi(f) = \mathrm{diag}(f \circ \xi_1, \ldots, f \circ \xi_k), 
\]
where $\xi_1, \ldots, \xi_k$ are continuous functions from $[0, 1]$ into $[0, 1]$, 
and $n_s^t = n_s^t(\xi)$ be as in Notation~\ref{nttn:_fs_and_nst}.  
Then the following are equivalent.  
\begin{enumerate}[label=\textup{(\roman*)}]
	\item There exists a unitary $u \in \mathcal{A}_{p'q'}$ such that 
		the image of $\mathrm{Ad}(u) \circ \varphi$ is included in $\mathcal{A}_{p'q'}$.  
	\item The congruence equations 
		\begin{equation}\label{eq:_arithmetic_condition}
		\begin{aligned}
			 qn_0^0 \equiv pn_0^1 \equiv 0 \pmod {q'}, 
			 \qquad qn_1^0 \equiv pn_1^1 \equiv 0 \pmod {p'}, \\
			 n_0^t \equiv 0 \pmod {q'}, \qquad 
			 n_1^t \equiv 0 \pmod {p'} \qquad (t \neq 0, 1)
		\end{aligned}
		\end{equation}
		hold.  
\end{enumerate}
Moreover, if $\mathcal{Z}_{p, q}$ is prime, 
then there exists a unitary $v \in \mathcal{A}_{p'q'}$ with the following property: 
for any $\psi \colon \mathcal{Z}_{p, q} \to \mathcal{A}_{p'q'}$ of the form 
\[
	\psi(f) = \mathrm{diag}(f \circ \zeta_1, \ldots, f \circ \zeta_k), 
\]
where $\zeta_1 \leq \cdots \leq \zeta_k$ are continuous functions from $[0, 1]$ into $[0, 1]$, 
if the numbers $n_s^t(\zeta)$ satisfies Eq.~(\ref{eq:_arithmetic_condition}), 
then the image of $\mathrm{Ad}(v) \circ \psi$ is included in $\mathcal{Z}_{p, q}$.  
\end{lem}


\begin{proof}
First, we shall prove (i) $\Rightarrow$ (ii).  
Let $F_s = F_s(\xi)$ be as in Notation~\ref{nttn:_fs_and_nst}.  
For $t \in F_0$, take $f^t \in \mathcal{Z}_{p, q}$ such that 
$f^t(t)$ is a minimal projection in $\mathrm{ev}_t[\mathcal{Z}_{p, q}]$ and 
$f^t(s)$ vanishes if $s \in F_0$.  
If $t \neq 0, 1$, then $\mathrm{ev}_0 \circ \varphi(f^t)$ is a projection of rank $n_0^t$ 
in $\mathbb{M}_{p'} \otimes 1_{\mathbb{M}_{q'}}$.  
Since the rank of any projection in $\mathbb{M}_{p'} \otimes 1_{\mathbb{M}_{q'}}$ is 
necessarily a multiple of $q'$, it follows that $n_0^t \equiv 0 \pmod {q'}$.  
On the other hand, $f^0(0)$ and $f^1(1)$ are minimal projections in 
$\mathbb{M}_p \otimes 1_{\mathbb{M}_q}$ and $1_{\mathbb{M}_p} \otimes \mathbb{M}_q$, 
so that their ranks are $q$ and $p$ respectively.  
Therefore, $\mathrm{ev}_0 \circ \varphi(f^0)$ and $\mathrm{ev}_0 \circ \varphi(f^1)$ are 
projections of ranks $qn_0^0$ and $pn_0^1$, 
which implies $qn_0^0 \equiv pn_0^1 \equiv 0 \pmod {q'}$.  
The other congruence equations in Eq.~(\ref{eq:_arithmetic_condition}) follow by similar arguments.  


Next, in order to see (ii) $\Rightarrow$ (i), suppose Eq.~(\ref{eq:_arithmetic_condition}) holds.  
If $\xi_i(0) = t$, then by definition of $n_s^t$, 
there are distinct suffixes $i = i_1, \ldots, i_{n_0^t}$ 
such that $\xi_{i_1}(0) = \cdots = \xi_{i_{n_0^t}}(0) = t$, 
so that the matrices $f \circ \xi_{i_1}(0), \ldots, f \circ \xi_{i_{n_0^t}}$ are 
equal for each $f \in \mathcal{Z}_{p, q}$.  
On one hand, for $t \neq 0, 1$, the number $n_0^t$ is a multiple of $q'$ by assumption; 
on the other hand, if $t = 0$ or $t = 1$, then 
$f \circ \xi_i(0)$ is included in $\mathbb{M}_p \otimes 1_{\mathbb{M}_q}$ or 
$1_{\mathbb{M}_p} \otimes \mathbb{M}_q$ respectively, 
so the congruence equation $qn_0^0 \equiv pn_0^1 \equiv 0 \pmod {q'}$ implies the existence of 
a permutation unitary $w$ such that  
$\mathrm{diag}(f \circ \xi_{i_1}(0), \ldots, f \circ \xi_{i_{n_0^t}})$ is 
equal to $\mathrm{Ad}(w)(a_f \otimes 1_{\mathbb{M}_{q'}})$ for some matrix $a_f$.  
Consequently, there is a permutation unitary $u_0 \in \mathbb{M}_{p'q'}$ 
such that the image of $\mathrm{Ad}(u_0) \circ \mathrm{ev}_0 \circ \varphi$ is included in 
$\mathbb{M}_{p'} \otimes 1_{\mathbb{M}_{q'}}$.  
Similarly, we can find a unitary $u_1 \in \mathbb{M}_{p'q'}$ such that 
the image of $\mathrm{Ad}(u_1) \circ \mathrm{ev}_1 \circ \varphi$ is included in 
$1_{\mathbb{M}_{q'}} \otimes \mathbb{M}_{p'}$.  
Since the unitary group of $\mathbb{M}_{p'q'}$ is path-connected, 
there is a unitary $u \in \mathcal{A}_{p'q'}$ with $u(0) = u_0$ and $u(1) = u_1$, 
so that the image of $\mathrm{Ad}(u) \circ \varphi$ is included in $\mathcal{Z}_{p', q'}$, 
as desired.  


Finally, suppose that $\mathcal{Z}_{p, q}$ is prime.  
Recalling the construction of the unitary $u$ in the preceding paragraph, 
and taking the assumption $\zeta_1 \leq \cdots \leq \zeta_k$ into account, 
we see that the existence of the unitary $v$ in the latter claim follows from 
the congruence equations 
\[
\begin{aligned}
	n_0^0(\xi) \equiv n_0^0(\zeta) \pmod {q'}, 
	\qquad n_0^1(\xi) \equiv n_0^1(\zeta) \pmod {q'}, \\
	n_1^0(\xi) \equiv n_1^0(\zeta) \pmod {p'}, 
	\qquad n_1^1(\xi) \equiv n_1^1(\zeta) \pmod {p'}.  
\end{aligned}
\]
By what we proved in the preceding paragraphs, we have 
\[
	qn_0^0(\xi) \equiv pn_0^1(\xi) \equiv 0 \pmod {q'}, 
	\qquad qn_0^0(\zeta) \equiv pn_0^1(\zeta) \equiv 0 \pmod {q'},   
\]
and 
\[
	n_0^0(\xi) + n_0^1(\xi) \equiv k \equiv n_0^0(\zeta) + n_0^0(\zeta) \pmod {q'}, 
\]
since $n_0^t \equiv 0 \pmod {q'}$ for $t \neq 0, 1$.  
Consequently, it follows that 
\[
\begin{aligned}
	p(n_0^0(\xi) - n_0^0(\zeta)) \equiv p(n_0^1(\xi) - n_0^1(\zeta)) \equiv 0 \pmod {q'}, \\
	q(n_0^1(\xi) - n_0^1(\zeta)) \equiv q(n_0^0(\xi) - n_0^0(\zeta)) \equiv 0 \pmod {q'}, 
\end{aligned}
\]
and so 
\[
	n_0^0(\xi) \equiv n_0^0(\zeta) \pmod {q'}, 
	\qquad n_0^1(\xi) \equiv n_0^1(\zeta) \pmod {q'}, 
\]
since $p$ and $q$ are co-prime.  The other equivalences follow similarly, 
which completes the proof.  
\end{proof}


We note that every trace on a dimension drop algebra bijectively corresponds to  
a probability Radon measure on $[0, 1]$, as in the case of $\mathcal{A}_n$.  
The following proposition is an immediate corollary of 
Lemma~\ref{lem:_measures_and_monotone_functions}.  
The proof is the same as Proposition~\ref{prop:_embeddability_lemma}, so we omit it.  

\begin{prop}\label{prop:_embeddability_lemma_2}
Let $\tau$ be a faithful trace on $\mathcal{Z}_{p, q}$.  
Then for any  faithful diffuse trace $\sigma$ on $\mathcal{Z}_{p, q}$, 
there is a morphism $\varphi \colon \langle \mathcal{Z}_{p, q}, \tau \rangle 
\to \langle \mathcal{Z}_{p, q}, \sigma \rangle$.   
\end{prop}


\begin{prop}\label{prop:_extension_lemma}
Let $p$ and $q$ be co-prime natural numbers.  
Then there exists $M(p, q) \in \mathbb{N}$ such that 
if $p'$ and $q'$ are co-prime natural numbers larger than $M(p, q)$ and if $pq$ divides $p'q'$, 
then for any faithful diffuse measures $\tau, \sigma$ on $[0, 1]$, 
we can find a morphism $\varphi$ from $\langle \mathcal{Z}_{p, q}, \tau \rangle$ into 
$\langle \mathcal{Z}_{p', q'}, \sigma \rangle$.  
\end{prop}


\begin{proof}
Let $c, d$ be divisors of $p, q$ respectively.  
Then, since $c$ and $d$ are co-prime, there exists $N(c, d) \in \mathbb{N}$ such that 
for any $n \in \mathbb{N}$, one can find $l, m \in \mathbb{N}$ with 
$(lc + md) < N(c, d)$ and $lc + md \equiv n \pmod {cd}$.  
We set 
\[
	M(p, q) := \max_{c|p, d|q} \frac{pq}{cd}N(c, d).  
\]


Now suppose that $p', q'$ are co-prime natural numbers larger than $M(p, q)$ and 
that $pq$ divides $p'q'$.  
Set $r := p'/(g_{p, p'}g_{q, p'})$ and $s := q'/(g_{p, q'}g_{q, q'})$, 
where $g_{n, m}$ denotes the greatest common divisor of $n$ and $m$.  
Note that since $p'$ and $q'$ are co-prime and $pq$ divides $p'q'$, 
the equations $p = g_{p, p'}g_{p, q'}$ and $q = g_{q, p'}g_{q, q'}$ hold.  
Since $r > M(p, q)/(g_{p, p'}g_{q, p'}) \geq N(g_{p, q'}, g_{q, q'})$ and 
similarly $s > N(g_{p, p'}, g_{q, q'})$, 
we can find $l_r, m_r, l_s, m_s \in \mathbb{N}$ 
such that both $r - l_r g_{p, q'} - m_r g_{q, q'}$ and $s - l_s g_{p, p'} - m_s g{q, p'}$ 
are positive and can be divided by $g_{p, q'}g_{q, q'}$ and $g_{p, p'}g_{q, p'}$ respectively.  
We shall put 
\[
\begin{aligned}
	a_0 := l_r g_{p, q'} s, \qquad b_0 := rs - a_0, \\
	a_1 := l_s g_{p, p'} r, \qquad b_1 := rs - a_1.  
\end{aligned}
\]

Suppose $a_0 > a_1$ and set$c := a_0-a_1$.  
We cut $[0, 1]$ into three intervals $I_1 = [0, t_1]$, $I_2 = [t_1, t_2]$ and $I_3 = [t_2, 1]$ 
so that 
\[
	\tau(I_1) = \frac{a_1 + 1/3}{rs}, \qquad 
	\tau(I_2) = \frac{c - 2/3}{rs}, \qquad 
	\tau(I_3) = \frac{b_0 + 1/3}{rs}.  
\]
Let $\tau_i$ be the normalization of $\tau|_{I_i}$.  
An argument similar to the proof of Lemma~\ref{lem:_measures_and_monotone_functions}
enable us to find continuous functions $\eta_1, \eta_2, \eta_3$ such that 
\begin{itemize}
	\item $\eta_1$ is a surjection from $[0, 1]$ onto $I_1$ 
		with $\eta_1(0) = \eta_1(1) = 0$ and $(\eta_1)_*(\sigma) = \tau_1$; 
	\item $\eta_2$ is the increasing surjection from $[0, 1]$ onto $[0, 1]$ 
		with $(\eta_2)_*(\sigma) = \frac{1}{3c}(\tau_1 + \tau_3) + (1 - \frac{2}{3c})\tau_2$; 
		and 
	\item $\eta_3$ is a surjection from $[0, 1]$ onto $I_3$ 
		with $\eta_3(0) = \eta_3(1) = 1$ and $(\eta_3)_*(\sigma) = \tau_3$.  
\end{itemize}
Then put
\[
	\xi_i := \left\{
		\begin{array}{ll}
			\eta_1 & \text{if } i = 1, \ldots, a_1, \\
			\eta_2 & \text{if } i = a_1 + 1, \ldots, a_0, \\
			\eta_3 & \text{if } i = a_0 + 1, \ldots, rs, 
		\end{array}
		\right.
\]
and consider the $*$-homomorphism $\varphi \colon \mathcal{Z}_{p, q} \to \mathcal{A}_{p'q'}$ 
defined by 
\[
	\varphi(f) = \mathrm{diag}(f \circ \xi_1, \ldots, f \circ \xi_{rs}).  
\]
It is not difficult to see from the definition of $\eta_i$ that $\varphi$ is trace-preserving.  
We shall check that this $\varphi$ satisfies the assumption of 
Lemma~\ref{lem:_arithmetic_condition}.  
Indeed, the functions $\eta_1, \eta_2, \eta_3$ are defined so that the equations
\[
	n_0^0 = a_0, \qquad n_0^1 = b_0, \qquad n_1^0 = a_1, \qquad n_1^1 = b_1
\]
holds, where $n_s^t = n_s^t(\xi)$ is as in Notation~\ref{nttn:_fs_and_nst}.  
Now, it follows that
\[
	q n_0^0 = q a_0 = q l_r g_{p, q'} s = g_{q, p'} l_r q' \equiv 0 \pmod {q'}, 
\]
and 
\[
	p n_0^1 = p (rs - a_0) = p (r - l_r g_{p, q'} - m_r g_{q, q'}) s + p m_r g_{q, q'} s 
	\equiv 0 \pmod {q'}.  
\]
The other congruences in Eq.~(\ref{eq:_arithmetic_condition}) can be similarly verified, 
so there exist a unitary $u \in \mathcal{A}_{p'q'}$ such that 
$\mathrm{Ad}(u) \circ \varphi$ is a morphism from $\langle \mathcal{Z}_{p, q}, \tau \rangle$ 
into $\langle \mathcal{Z}_{p', q'}, \sigma \rangle$, which completes the proof.  
\end{proof}


\begin{cor}
The class $\mathscr{K}$ satisfies JEP.  
\end{cor}


\begin{proof}
Let $\langle \mathcal{Z}_{p_1, q_1}, \tau_1 \rangle$, 
$\langle \mathcal{Z}_{p_2, q_2}, \tau_2 \rangle$ be members of $\mathscr{K}$.  
Find co-prime $p_3, q_3 \in \mathbb{N}$ such that both $p_3$ and $q_3$ are larger than 
$\max \{M(p_1, q_1), M(p_2, q_2)\}$ and $p_3 q_3$ is divided by $\mathrm{lcm}(p_1 q_1, p_2 q_2)$.  
Then by Propositions~\ref{prop:_embeddability_lemma_2} and 
\ref{prop:_extension_lemma}, there is a morphism $\varphi_i$ from 
$\langle \mathcal{Z}_{p_i, q_i}, \tau_i \rangle$ into 
$\langle \mathcal{Z}_{p_3, q_3}, \tau_3 \rangle$, where $\tau_3$ is a diffuse faithful trace.  
\end{proof}


\begin{prop}\label{prop:_diagonalizability_2}
Let $\varphi \colon \langle \mathcal{Z}_{p, q}, \tau \rangle \to 
\langle \mathcal{Z}_{p', q'}, \sigma \rangle$ be a morphism.  
Then for any finite subset $G \subseteq \mathcal{Z}_{p, q}$ and any $\varepsilon > 0$, 
there exists a diagonalizable morphism $\psi \colon \langle \mathcal{Z}_{p, q}, \tau \rangle 
\to \langle \mathcal{Z}_{p', q'}, \sigma \rangle$ with $\|\varphi(g) - \psi(g)\| < \varepsilon$ 
for all $g \in G$.  
Moreover, we can take $\psi$ so that 
the maps $\xi_1, \ldots, \xi_k$ associated to a diagonal expression of $\psi$ 
satisfies $\xi_1 \leq \cdots \leq \xi_k$.  
\end{prop}


\begin{proof}
It can be easily seen that the proof of Proposition~\ref{prop:_diagonalizability_1} works 
even if $\mathcal{A}_n$ and $\mathcal{A}_m$ are replaced by 
$\mathcal{Z}_{p, q}$ and $\mathcal{A}_{p'q'}$ respectively, 
so one can easily obtain a morphism $\psi$ from $\langle \mathcal{Z}_{p, q}, \tau \rangle$ 
into $\langle \mathcal{A}_{p'q'}, \sigma \rangle$ with 
$\|\varphi(g) - \psi(g)\| < \varepsilon$ for all $g \in G$.  
Moreover, a careful reading and a trivial modification of the third paragraph of 
the proof of Proposition~\ref{prop:_diagonalizability_1} enable us to take the unitary $u$ 
so that $\mathrm{ev}_0 \circ \varphi(f) = \mathrm{ev}_0 \circ \psi(f)$ and 
$\mathrm{ev}_1 \circ \varphi(f) = \mathrm{ev}_1 \circ \psi(f)$ hold for all 
$f \in \mathcal{Z}_{p, q}$.  
Therefore, we can take $\psi$ so that its image is included in $\mathcal{Z}_{p', q'}$, 
which completes the proof.  
\end{proof}


\begin{prop}\label{prop:_cutting-up_lamma_2}
Let $\tau, \sigma$ be faithful diffuse measures on $[0, 1]$.  
Then for any co-prime $p, q \in \mathbb{N}$, 
there exist co-prime $p', q' \in \mathbb{N}$ and a diagonalizable morphism 
$\varphi \colon \langle \mathcal{Z}_{p, q}, \tau \rangle 
\to \langle \mathcal{Z}_{p', q'}, \sigma \rangle$ such that the images of 
the maps associated to diagonal expression of $\varphi$ have diameters less than $\varepsilon$.  
\end{prop}


\begin{proof}
The proof is similar to that of Proposition~\ref{prop:_cutting-up_lamma}, 
but this time, instead of dividing $[0, 1]$ into two intervals $[0, t_0]$ and $[t_0, 1]$, 
we divide $[0, 1]$ into three intervals $[0, t_0]$, $[t_0, t_1]$ and $[t_1, 1]$ 
the diameters of which are close to $1/3$, and use Lemma~\ref{lem:_arithmetic_condition}.  


Take integers $l > 2q$, $m > 2p$ so that $p_1 := lp$ and $q_1 := mq$ are co-prime, 
and set $k := lm$.  Then let $r, s$ be natural numbers such that 
\[
	r \equiv k \pmod {q_1}, \qquad s \equiv k \pmod {p_1}, \qquad r+s < k.  
\]
We can always find such $r$ and $s$ because $k = lm > \max\{2p_1, 2q_1\}$.  
Since $\tau$ is diffuse, there are $t_0 < t_1$ in $(0, 1)$ such that 
$\tau([0, t_0]) = r/k$ and $\tau([t_1, 1]) = s/k$.  
Here, we may assume that the diameters of $[0, t_1]$ and $[t_2, 1]$ are arbitrarily close to $1/3$, 
because $l$ and $m$ can be arbitrarily large and so $q_1/k$ and $p_1/k$ can be arbitrarily small.  


Now, put $I_1 := [0, t_0]$, $I_2 := [t_0, t_1]$ and $I_3 := [t_1, 1]$, 
and let $\tau_i$ be the normalization of $\tau|_{I_i}$.  
By Lemma~\ref{lem:_measures_and_monotone_functions}, 
we can find continuous function $\eta_i$ from $[0, 1]$ onto $I_i$ such that 
\begin{itemize}
	\item $\eta_1$ and $\eta_3$ are increasing; 
	\item $\eta_2$ is decreasing; and
	\item $(\eta_i)_*(\sigma) = \tau_i$.  
\end{itemize}
We shall set 
\[
	\xi_j := \left\{
		\begin{array}{ll}
			\eta_1 & \text{if } j = 1, \ldots, r, \\
			\eta_2 & \text{if } j = r+1, \ldots, k-s, \\
			\eta_3 & \text{if } j = k-s+1, \ldots, k, 
		\end{array}
		\right.
\]
and check the assumption of Lemma~\ref{lem:_arithmetic_condition}.  
Let $n_s^t = n_s^t(\xi)$ be as in Notation~\ref{nttn:_fs_and_nst}.  
Then clearly $n_0^1 = n_1^0 = 0$.  Also, 
\[
	q n_0^0 = qr \equiv qk = q_1 l \equiv 0 \pmod {q_1}, 
\] 
and 
\[
	p n_1^1 = ps \equiv pk = p_1 m \equiv 0 \pmod {p_1}, 
\]
so Eq.~(\ref{eq:_arithmetic_condition}) holds, as desired.  
Consequently, there is a diagonalizable morphism 
$\varphi_1 \colon \langle \mathcal{Z}_{p, q}, \tau \rangle 
\to \langle \mathcal{Z}_{p', q'}, \sigma \rangle$ such that the ranges of
the maps associated to $\varphi_1$ have their diameters arbitrarily close to $1/3$.  
The remaining of the proof is now the same as Proposition~\ref{prop:_cutting-up_lamma}, 
so we are done.  
\end{proof}


\begin{lem}\label{lem:_automorphism_of_dimension_drop_algebras}
Let $\varphi, \psi$ be $*$-homomorphisms from $\mathcal{Z}_{p, q}$ into $\mathcal{Z}_{p', q'}$ 
of the form 
\[
\begin{aligned}
	\varphi(f) = \mathrm{Ad}(u) 
		\bigl( \mathrm{diag}(f \circ \xi_1, \ldots, f \circ \xi_k) \bigr), \\
	\psi(f) = \mathrm{Ad}(v) 
		\bigl( \mathrm{diag}(f \circ \xi_1, \ldots, f \circ \xi_k) \bigr).  
\end{aligned}
\]
Then, for any finite subset $G \subseteq \mathcal{Z}_{p, q}$ and any $\varepsilon > 0$, 
there exists a unitary $w \in \mathcal{A}_{p'q'}$ such that 
the inner automorphism $\mathrm{Ad}(w)$ of $\mathcal{A}_{p'q'}$ preserves $\mathcal{Z}_{p', q'}$ 
and $\|\mathrm{Ad}(w) \circ \varphi(g) - \psi(g)\| < \varepsilon$ holds for all $g \in G$.  
\end{lem}


\begin{proof}
Let $\rho$ be a the $*$-homomorphism from $\mathcal{Z}_{p, q}$ into $\mathcal{A}_{p'q'}$ defined by
\[
	\rho(f) = \mathrm{diag}(f \circ \xi_1, \ldots, f \circ \xi_k),  
\]
and put 
\[
\begin{aligned}
	\mathcal{B}_s &:= \mathrm{ev}_s \circ \rho[\mathcal{Z}_{p, q}], \\
	\mathcal{C}_s^\varphi &:= \mathrm{ev}_s \circ \mathrm{Ad}(u^*)[\mathcal{Z}_{p', q'}] 
	\quad \text{and} \quad 
	\mathcal{C}_s^\psi := \mathrm{ev}_s \circ \mathrm{Ad}(v^*)[\mathcal{Z}_{p', q'}]
\end{aligned}
\]
for $s = 0, 1$.  
Then, $\mathcal{C}_s^\varphi$ and $\mathcal{C}_s^\psi$ are subalgebras of $\mathbb{M}_{p'q'}$ 
which are isomorphic to each other, and $\mathcal{B}_s$ is included in both of them.  
It is not difficult to find a unitary $w'_s$ in the commutant $(\mathcal{B}_s)'$ 
of $\mathcal{B}_s$ which induces the isomorphism of 
$\mathcal{C}_s^\varphi$ onto $\mathcal{C}_s^\psi$.  


Now, take $\delta > 0$ so that $|t_1-t_2| < \delta$ implies 
$\|g \circ \xi_i(t_1) - g \circ \xi_i(t_2)\| < \varepsilon/2$, 
and let $w'$ be a unitary in $\mathcal{A}_{p'q'}$ such that 
\begin{itemize}
	\item $w'(0) = w'_0$ and $w'(1) = w'_1$; 
	\item $w'(t) = 1$ for $t \in [\delta, 1-\delta]$; and 
	\item the images of $w'|_{[0, \delta]}$ and $w'|_{[1-\delta, 1]}$ is included in 
		$(\mathcal{B}_0)'$ and $(\mathcal{B}_1)'$ respectively.  
\end{itemize}
Then, clearly the inner automorphism induced by $w := vw'u^*$ preserves $\mathcal{Z}_{p', q'}$.  
Also, for $g \in G$ and $t \in [0, \delta]$, 
\begin{align*}
	\mathrm{ev}_t \circ \mathrm{Ad}(w) \circ \varphi(g)
		&\ = \ \mathrm{Ad}(v(t)w'(t)) 
			\bigl( \mathrm{diag}(g \circ \xi_1(t), \ldots, g \circ \xi_k(t)) \bigr) \\
		&\sim_{\varepsilon/2} \mathrm{Ad}(v(t)w'(t)) 
			\bigl( \mathrm{diag}(g \circ \xi_1(0), \ldots, g \circ \xi_k(0)) \bigr) \\
		&\ = \ \mathrm{Ad}(v(t)) 
			\bigl( \mathrm{diag}(g \circ \xi_1(0), \ldots, g \circ \xi_k(0)) \bigr) \\
		&\sim_{\varepsilon/2} \mathrm{Ad}(v(t)) 
			\bigl( \mathrm{diag}(g \circ \xi_1(t), \ldots, g \circ \xi_k(t)) \bigr) \\
		&\ = \ \mathrm{ev}_{t} \circ \psi(g).  
\end{align*}
Similarly, it follows that $\mathrm{ev}_t \circ \mathrm{Ad}(w) \circ \varphi(g) \sim_\varepsilon 
\mathrm{ev}_t \circ \psi(g)$ if $t$ is in $[1-\delta, 1]$, and 
it is obvious that $\mathrm{ev}_t \circ \mathrm{Ad}(w) \circ \varphi(g) = 
\mathrm{ev}_t \circ \psi(g)$ if $t$ is in $[\delta, 1-\delta]$.  
Consequently, $\|\mathrm{Ad}(w) \circ \varphi(g) - \psi(g)\|$ is less than $\varepsilon$ for 
all $g \in G$, which completes the proof.  
\end{proof}


\begin{prop}
The class $\mathscr{K}$ satisfies NAP.   
\end{prop}


\begin{proof}
Let $\varphi_1, \varphi_2 \colon \langle \mathcal{Z}_{p_0, q_0}, \tau_0 \rangle \to  
\langle \mathcal{Z}_{p_1, q_1}, \tau_1 \rangle$ be diagonalizable morphisms, 
$G$ be a finite subset included in the unit ball of $\mathcal{Z}_{p, q}$, 
and $\varepsilon$ be a positive real number.  
By JEP, it suffices to find morphisms $\psi_1, \psi_2$ from 
$\langle \mathcal{Z}_{p_1, q_1}, \tau_1 \rangle$ into 
some $\langle \mathcal{Z}_{p_2, q_2}, \tau_2 \rangle \in \mathscr{K}$ such that 
$\|\psi_1 \circ \varphi_1(g) - \psi_2 \circ \varphi_2(g)\|<\varepsilon$ holds for all $g \in G$.  


Take $\delta > 0$ so that $|s-t|<\delta$ implies $\|g(s)-g(t)\|<\varepsilon$.  
As in the proof of Proposition~\ref{prop:_nap_for_matrix-valued_function_algebras}, 
we can find a diagonalizable morphism $\rho$ from $\langle \mathcal{Z}_{p_1, q_1}, \tau_1 \rangle$ 
into some $\langle \mathcal{Z}_{p_2, q_2}, \tau_2 \rangle \in \mathscr{K}$, 
and diagonalizable morphisms $\Phi_1, \Phi_2$ from 
$\langle \mathcal{Z}_{p_0, q_0}, \tau_0 \rangle$ into 
$\langle \mathcal{Z}_{p_2, q_2}, \tau_2 \rangle$ with the following properties: 
\begin{itemize}
	\item the inequality $\| \rho \circ \varphi_i(g) - \Phi_i(g) \| < \varepsilon$ holds 
		for all $g \in G$; and
	\item there is a diagonal expression 
		\[
			\Phi_i(f) = \mathrm{Ad}(u_i) 
				\bigl( \mathrm{diag}(f \circ \xi^i_1, \ldots, f \circ \xi^i_k) \bigr) 
		\]
		such that $\xi^i_1 \leq \cdots \leq \xi^i_k$ for each $i$, 
		and $\|\xi^1_j - \xi^2_j\|<\delta$ for all $j$.  
\end{itemize}
By Lemma~\ref{lem:_arithmetic_condition}, 
there exists a unitary $v \in \mathcal{A}_{p'q'}$ such that the image of 
$\Psi_i := \mathrm{Ad}(vu_i^*) \circ \Phi_i$ is included in $\mathcal{Z}_{p', q'}$.  
Then by Lemma~\ref{lem:_automorphism_of_dimension_drop_algebras}, 
there exists a unitary $w_i \in \mathcal{A}_{p'q'}$ such that 
the inner automorphism $\mathrm{Ad}(w_i)$ preserves $\mathcal{Z}_{p', q'}$, 
and that $\|\mathrm{Ad}(w_i) \circ \Phi_i(g) - \Psi_i(g)\| < \varepsilon$ holds for all $g \in G$.  
We put $\psi_1 := \rho$ and $\psi_2 := \mathrm{Ad}(w_1^*w_2) \circ \rho$.  
Then for $g \in G$, we have 
\begin{align*}
	\psi_2 \circ \varphi_2(g) 
		&= \mathrm{Ad}(w_1^*w_2) \circ \rho \circ \varphi_2(g) \\
		&\sim_\varepsilon \mathrm{Ad}(w_1^*w_2) \circ \Phi_2(g) \\
		&\sim_\varepsilon \mathrm{Ad}(w_1^*) \circ \Psi_2(g) \\
		&= \mathrm{Ad}(w_1^*v) 
			\bigl( \mathrm{diag}(g \circ \xi^2_1, \ldots, g \circ \xi^2_k) \bigr) \\
		&\sim_\varepsilon \mathrm{Ad}(w_1^*v) 
			\bigl( \mathrm{diag}(g \circ \xi^1_1, \ldots, g \circ \xi^1_k) \bigr) \\
		&= \mathrm{Ad}(w_1^*) \circ \Psi_1(g) \\
		&\sim_\varepsilon \Phi_1(g) \\
		&\sim_\varepsilon \psi_1 \circ \varphi_1(g), 
\end{align*}
which completes the proof.  
\end{proof}


The following theorem can be shown in almost the same way as 
Theorem~\ref{thm:_knu_is_a_fraisse_class}.  We omit details.  

\begin{thm}
The class $\mathscr{K}$ is a Fra\"{i}ss\'{e} class.  
\end{thm}


We close this section by showing that the Fra\"{i}ss\'{e} limit of 
$\mathscr{K}$ is simple and monotracial.  
This fact together with \cite[Theorem 6.2]{jiang99:_simple_unital} implies that 
the Fra\"{i}ss\'{e} limit is indeed the Jiang-Su algebra.  


\begin{lem}\label{lem:_trace_and_morphisms}
For a measure $\tau$ on $[0, 1]$, let $E(\tau)$ be the set of all morphisms 
from $\langle \mathcal{Z}_{1, 1}, \tau \rangle$ into 
some $\langle \mathcal{Z}_{p, q}, \tau' \rangle \in \mathscr{K}$.  
If $\tau$ is diffuse and faithful, and if $\sigma$ is a measure 
with $E(\sigma) \supseteq E(\tau)$, then $\sigma = \tau$.  
\end{lem}


\begin{proof}
Suppose $\sigma \neq \tau$.  
Then there exists $s \in (0, 1)$ with $\sigma([0, s]) \neq \tau([0, s])$.  
If $\sigma([0, s]) > \tau([0, s])$, then since $\tau$ is diffuse 
there exists $t > s$ such that $\sigma([0, s]) = \tau([0, t])$.  
Now apply Proposition~\ref{prop:_cutting-up_lamma_2} to find a diagonalizable morphism 
$\varphi \colon \langle \mathcal{Z}_{1, 1}, \tau \rangle 
\to \langle \mathcal{Z}_{p, q}, \tau \rangle$ such that the images of 
the maps $\xi_1, \ldots, \xi_k$ associated to a diagonal expression of $\varphi$ have
diameters less than $(t-s)/3$, and set 
\[
	S = \{\xi_i \mid \mathrm{Im} \xi_i \cap [0, s] \neq \varnothing\}, \qquad 
	T = \{\xi_i \mid \mathrm{Im} \xi_i \subseteq [0, t] \}.
\]
Then clearly $S \subsetneq T$, and since $\varphi \in E(\tau) \subseteq E(\sigma)$, 
it follows that
\[
	\sigma([0, s]) \leq \frac{\#S}{pq} < \frac{\#T}{pq} \leq \tau([0, t]) = \sigma([0, s]), 
\]
which is a contradiction.  
The inequality $\sigma([0, s]) < \tau([0, s])$ implies a similar contradiction, 
so $\sigma = \tau$.   
\end{proof}


\begin{prop}\label{prop:_jiang-su_algebra_is_monotracial}
Let $\langle \mathcal{Z}, \tau \rangle$ be the Fra\"{i}ss\'{e} limit of the class $\mathscr{K}$.  
Then $\tau$ is the unique trace on $\mathcal{Z}$.  
\end{prop}


\begin{proof}
Let $\sigma$ be a faithful trace on $\mathcal{Z}$ and suppose $\sigma \neq \tau$.  
Then there exists an embedding $\iota \colon \mathcal{Z}_{1, 1} \to \mathcal{Z}$ such that 
$\iota^*(\tau) \neq \iota^*(\sigma)$.  
Now assume $\varphi \colon \langle \mathcal{Z}_{1, 1}, \iota^*(\tau) \rangle 
\to \langle \mathcal{Z}_{p, q}, \rho \rangle$ is 
in $E(\iota^*(\tau))$, and find a morphism 
$\psi \colon \langle \mathcal{Z}_{p, q}, \rho \rangle \to \langle \mathcal{Z}, \tau \rangle$.  
Since $\langle \mathcal{Z}, \tau \rangle$ is approximately $\mathscr{K}$-homogeneous, 
for any finite $G \subseteq \mathcal{Z}_{1, 1}$ and any $\varepsilon > 0$ 
there exists an automorphism $\alpha_{G, \varepsilon}$ of 
$\langle \mathcal{Z}, \tau \rangle$ such that 
$\|\alpha \circ \psi \circ \varphi (g) - \iota(g)\| < \varepsilon$ for all $g \in G$.  
Set $\rho'_{G, \varepsilon} := (\psi \circ \alpha_{G, \varepsilon})^*(\sigma)$ and let $\rho'$ be 
the limit of $\{\rho'_{G, \varepsilon}\}$ in $T(\mathcal{Z}_{p, q})$.  
Then it is clear that $\varphi \colon \langle \mathcal{Z}_{1, 1}, \iota^*(\sigma) \rangle \to 
\langle \mathcal{Z}_{p, q}, \rho' \rangle$ is trace-preserving, 
so $\varphi \in E(\iota^*(\sigma))$, which is a contradiction.  
\end{proof}


\begin{prop}\label{prop:smp}
The Fra\"{i}ss\'{e} limit $\langle \mathcal{Z}, \tau \rangle$ is simple.  
\end{prop}


\begin{proof}
Suppose $\mathcal{I} \subseteq \mathcal{Z}$ be a non-trivial ideal.  
For each embedding $\iota \colon \mathcal{Z}_{1, 1} \to \mathcal{Z}$, 
let $\Sigma_\iota$ be the closed subset of $[0, 1]$ which corresponds to the ideal 
$\mathcal{I}_\iota := \mathcal{Z}_{1, 1} \cap \iota^{-1}[\mathcal{I}]$.   
For each $\varepsilon > 0$, choose a function $f_\iota^\varepsilon \in \mathcal{I}_\iota$ such that 
$|f(t)| \geq \varepsilon$ holds if $\mathrm{dist}(t, \Sigma_\iota) \geq \varepsilon$.  
Now, by Proposition~\ref{prop:_cutting-up_lamma_2}, there exists a diagonalizable morphism 
$\varphi$ such that the images of the maps $\xi_1, \ldots, \xi_k$ associated to 
a diagonal expression of $\varphi$ have diameters less than $\varepsilon$, 
and as in the proof of Proposition~\ref{prop:_jiang-su_algebra_is_monotracial}, 
we can find an embedding $\eta \colon \mathcal{Z}_{1, 1} \to \mathcal{Z}$ 
which factors through $\varphi$ and satisfies 
$\|\iota(f_\iota^\varepsilon) - \eta(f_\iota^\varepsilon)\| < \varepsilon$.  
Then $\Sigma_\eta$ is included in the $\varepsilon$-neighborhood of $\Sigma_\iota$, 
since $\mathrm{dist}(f_\iota^\varepsilon, \mathcal{I}_\eta) 
= \mathrm{dist}(\eta(f_\iota^\varepsilon), \mathcal{I}) < \varepsilon$ 
(see the proof of \cite[Lemma {III}.4.1]{davidson96:_cstar_algebras}, for example).  
On the other hand, clearly $\Sigma_\eta \cap \mathrm{Im}\xi_i$ is non-empty for all $i$, 
so $\Sigma_\iota$ intersects with every $4\varepsilon$-ball.  
Since this is true for any $\varepsilon > 0$ and any $\iota$, 
it follows that $\Sigma_\iota$ is equal to $[0, 1]$ for any $\iota$, so $\mathcal{I} = 0$.  
\end{proof}


\bigskip
\noindent
\textbf{Acknowledgement.}
The author wishes to express his thanks to Professor Ilijas Farah and Professor George Elliott
for many stimulating conversations and helpful advices, 
and Alessandro Vignati for many corrections.  
The author also would like to thank the Fields Institute for Research 
in Mathematical Sciences, where the paper is written, for the hospitality.  
This work was supported by Research Fellow of the JSPS (no.~26-2990) and 
the Program for Leading Graduate Schools, MEXT, Japan.



\begin{thebibliography}{Eag+15}
\bibitem[Yaa15]{yaacov15:_fraisse_limits}
	Ita\"{i} Ben Yaacov, 
	\emph{Fra\"{i}ss\'{e} Limits of Metric Structures.}\/
	Journal of Symbolic Logic \textbf{80} (2015), no.~1, 100--115.  
\bibitem[Dav96]{davidson96:_cstar_algebras}
	Kunneth R.~Davidson, 
	\emph{C*-Algebras by Example.}\/
	Fields Institute Monographs, \textbf{6}.  
	American Mathematical Society, 1996. 
\bibitem[Eag+15]{eagle15:_fraisse_limits}
	C.~J.~Eagle, I.~Farah, B.~Hart, B.~Kadets, V.~Kalashunyk and M.~Lupini, 
	\emph{Fra\"{i}ss\'{e} Limits of C*-Algebras.}\/
	arXiv: \texttt{1411.4066v3} (2015, preprint).  
\bibitem[ET08]{elliott08:_regularity_properties}
	G.~A.~Elliott and A.~S.~Tom, 
	\emph{Regularity Properties in the Classification Program 
	for Separable Amenable C*-Algebras.}\/
	Bulletin of the American Mathematical Society \textbf{45} (2008), no.~2, 229--245.  
\bibitem[Far14]{farah14:_logic_operator}
	Ilijas Farah, 
	\emph{Logic and Operator Algebras.}\/
	Proceedings of the Seoul ICM, 2014. 
	arXiv: \texttt{1404.4978}
\bibitem[Fra54]{fraisse54:_extension_relations}
	Roland Fra\"{i}ss\'{e}, 
	\emph{Sur l'extension aux relations de quelques propri\'{e}t\'{e}s des orderes.}\/
	Annales Scientifiques de l'\'{E}cole Normale Sup\'{e}rieure, S\'{e}r. 3, 
	\textbf{71} no.~4 (1954), 363--388.  
\bibitem[Gli60]{glimm60:_certain_class}
	James G.~Glimm, 
	\emph{On a Certain Class of C*-algebras.}\/
	Transactions of the American Mathematical Society \textbf{95} (1960), 318--340.  
\bibitem[JS99]{jiang99:_simple_unital}
	X.~Jiang and H.~Su, 
	\emph{On a Simple Unital Projectionless C*-Algebra.}\/
	American Journal of Mathematics \textbf{121} (1999), no.~2, 359--413.  
\bibitem[KS13]{kubis13:_proof_uniqueness}
	W.~Kubi\'{s} and S.~Solecki, 
	\emph{A Proof of the Uniqueness of the Gurari\u{i} Space.}\/
	Israel Journal of Mathematics \textbf{195} (2013), no.~1, 449--456.  
\bibitem[Lup14]{lupini14:_uniqueness_universality}
	Martino Lupini, 
	\emph{Uniqueness, Universality, and Homogeneity of the Noncommutative Gurarij Space.}\/
	arXiv: \texttt{1410.3345v3} (2014, preprint).  
\bibitem[Sch07]{schoretsanitis07:_fraisse_theory}
	Konstantinos Schoretsanitis, 
	\emph{Fra\"{i}ss\'{e} Theory for Metric Structures.}\/
	PhD thesis, the Graduate College of the University of Illinois at Urbana-Champaign, 2007.  
\end{thebibliography}
\end{document}